\documentclass[a4paper, 11pt, oneside, headsepline, ngerman]{article}
\usepackage[T1]{fontenc}
\usepackage[utf8,ansinew]{inputenc}
\usepackage{scrpage2}
\usepackage{enumerate}
\usepackage{mathtools}
\usepackage{amssymb,multicol,multirow}
\usepackage{amsthm}
\usepackage[T1]{fontenc}
\usepackage[all]{xy}
\usepackage{verbatim}
\usepackage{hhline}
\newtheorem{theorem}{Theorem}
\newtheorem{lemma}{Lemma}
\newtheorem{proposition}{Proposition}
\newtheorem{definition}{Definition}
\newtheorem{corollary}{Corollary}

\usepackage[pdftex]{graphicx}
\usepackage[intoc]{nomencl}
\makenomenclature
\usepackage[numbers]{natbib}
\usepackage{pdfpages}
\usepackage[vcentering,dvips]{geometry}
\usepackage{color}
\usepackage{url}
\usepackage{hyperref}
\hypersetup{colorlinks=true, linkcolor=blue,urlcolor=blue, citecolor=blue}
\usepackage{caption}
\usepackage{fancybox}
\usepackage{longtable}
\usepackage{tabularx}
\usepackage{eurosym}
\newcolumntype{C}[1]{>{\centering\arraybackslash}p{#1}}
\newcolumntype{R}[1]{>{\raggedleft\arraybackslash}p{#1}}

\title{Relating the spectrum of a matrix and a principal submatrix using adjugates and Schur complements}
\author{Mario Thüne}
\date{}
\begin{document}
\maketitle
\begin{abstract}
Let $\mathcal{M}$ be a square matrix over a commutative ring and let $\mathcal{A}$ be a principal submatrix. We give relations between the determinants of $\mathcal{M}$ and $\mathcal{A}$ based on an annihilating polynomial for one of them. The intended application is the size reduction of complex latent root problems, especially the reduction of ordinary eigenvalue problems if a matrix or its principal submatrix have a low degree minimal polynomial. An example is the spectrum of vertex perturbed strongly regular graphs.
\end{abstract}
\section{Introduction}
\emph{Let $\mathbf{M}$ be a complex matrix with principal submatrix $\mathbf{A}$. Given the spectrum of $\mathbf{A}$ (or $\mathbf{M}$), how can that be use to determine the spectrum of $\mathbf{M}$ (of $\mathbf{A}$)?}

Since the eigenvalues of a complex matrix $\mathbf{Z}_0$ are the roots of $\det\left(\lambda\mathbf{I}-\mathbf{Z}_0\right)$ and roots of its annihilating polynomials, we put this problem in a more general setting. Considering a square matrix $\mathcal{Z}$ over a commutative ring $R$, we relate $\det\left(\mathcal{Z}\right)$ to the determinant of a principal super- (or sub-) matrix using a variation on Schur complements, assuming that $f\in R$, two square matrices $\mathcal{X}$ and $\mathcal{Y}$ and a polynomial $\mathrm{\alpha}\left(x\right)\in R\left[x\right]$ are available s.t. $\mathrm{\alpha}\left(f\mathbf{I}-\mathcal{X}\mathcal{Z}\mathcal{Y}\right)=0.$ The problem above arises by $R=\mathbb{C}\left[\lambda\right]$,  $\mathcal{X}=\mathcal{Y}=\mathbf{I}$, $\mathcal{Z}=\lambda\mathbf{I}-\mathcal{Z}_0$ with $\mathcal{Z}_0\in\left\{\mathbf{M},\mathbf{A}\right\}$ and $\mathrm{\alpha}\left(x\right)\in\mathbb{C}\left[x\right]\subset\left(\mathbb{C}\left[\lambda\right]\right)\left[x\right]$ with $\mathrm{\alpha}\left(\mathcal{Z}_0\right)=\mathbf{0}$. 

Let $\mathbf{M}$ be of size $N$, $\mathbf{A}$ be of size $n=N-s$ and let $\mathrm{\alpha}$ have degree $d$. As will be shown, we can basically reduce the eigenproblem of $\mathbf{M}$ (of $\mathbf{A}$) to a monic polynomial eigenproblem of size $s$, which, by moniticity, can be transformed into an ordinary eigenproblem of size $s\left(d+1\right)$ (resp. $s\left(d-1\right)$), i.e. if $s$ and $d$ are small enough a size reduction can be achieved even within the framework of ordinary eigenvalue problems.
\paragraph{Structure of the Article}
Section~\ref{sec_AMSC} collects several elementary results for adjugates and Schur-like complements of matrices over a commutative ring R, which are well known for the case $R=\mathbb{C}$.
Our basic theorems are given in section~\ref{sec_BT}. Section~\ref{sec_ISR} considers our method as a generalization to isospectral graph reductions~\cite{BUNIMOVICH20121429} developed for matrices over meromorphic functions. It provides a simple way of transforming latent pairs of the reduced matrix to those of the initial matrix, which generalizes the corresponding result in~\cite{Duarte2015110}. We also show that in the generalized framework the reduced matrix still provides an improved spectral approximation as considered in~\cite{BUNIMOVICH20121429}.
Section~\ref{sec_AOEP} discusses the special case of ordinary complex eigenvalue problems followed by an example application to vertex perturbations of strongly regular graphs in section~\ref{sec_VPSRG}.
\section{Adjugate-like Matrix and Schur-like Complement}\label{sec_AMSC}
Let $R$ be a commutative ring. In order to simplify some proofs let $R$ be unital. An element $r$ is \emph{regular} if it is not zero and not a zero divisor. A square matrix is \emph{regular} if its determinant is regular. The following proposition is well known.
\begin{proposition}\label{pro_reg}
Let $\mathcal{X},\mathcal{Y}\in R^{n\times n}$, $\mathcal{X}$ regular and $\mathcal{X}\mathcal{Y}=r\mathbf{I}_n$ with $r\in R$. $\mathcal{Y}$ is regular if and only if $r$ is regular and $\mathcal{Y}=\mathbf{0}$ if and only if $r=0$.
\end{proposition}
\subsection{Adjugate-like Matrices}\label{subsec_AM}
\begin{definition}\label{def_pk}
$\mathrm{p}_k\left(y,z\right)=\sum_{i=1}^{k}y^{i-1}z^{k-i}\in R\left[y,z\right]$.
\end{definition}
The identity $y\hspace{1pt}\mathrm{p}_k\left(y,z\right)-\mathrm{p}_k\left(y,z\right)z=y^k-z^k$ is easy to see. If the indeterminates $y$ and $z$ commute, i.e. $yz=zy$, we have
$\left(y-z\right)\mathrm{p}_k\left(y,z\right)=y^k-z^k$.

Let $\mathrm{a}\left(x\right)=\sum_{k=0}^{d}a_kx^k$ be a polynomial over $R$. We define
\begin{definition}\label{def_p}
$\mathrm{p}\left(y,z;\mathrm{a}\right)=\sum_{k=1}^{d}a_k\mathrm{p}_k\left(y,z\right)$.
\end{definition}
Again assuming commuting $y$ and $z$ we have
\begin{equation}\label{eq_(y-z)p=qy-qz}
\left(y-z\right)\mathrm{p}\left(y,z;\mathrm{a}\right)=\mathrm{a}\left(y\right)-\mathrm{a}\left(z\right),
\end{equation}
known as the \emph{polynomial remainder theorem}. We are interested in the case where $y=f\mathbf{I}_n$ with $f\in R$ and $z=\mathcal{X}\in R^{n\times n}$ (cf.~\cite{G60TheMat}).
\begin{definition}\label{def_adj}
Let $\mathcal{X}\hspace{-1pt}\in\hspace{-1pt}R^{n\times n}$ and $\mathrm{a}\hspace{-1pt}\left(x\right)\hspace{-1pt}=\hspace{-1pt}\sum_{k=0}^{d}a_kx^k\in R\left[x\right]\setminus{0}$ s.t. $\mathrm{a}\hspace{-1pt}\left(\mathcal{X}\right)\hspace{-1pt}=\hspace{-1pt}\mathbf{0}$.
We define the \emph{adjugate-like matrix} of $\mathcal{X}$ corresponding to $\mathrm{a}\left(x\right)$ as
\begin{equation}\label{eq_adj}
\mathrm{p}\left(\lambda\mathbf{I}_n,\mathcal{X};\mathrm{a}\right)
=\textstyle{\text{$\sum_{k=1}^{d}$}}a_k\textstyle{\text{$\sum_{i=1}^{k}$}}\lambda^{i-1}\mathcal{X}^{k-i}.
\end{equation}
\end{definition}
If $\mathrm{a}\left(x\right)$ is an annihilating polynomial of $\mathcal{X}$, then we have, by~\eqref{eq_(y-z)p=qy-qz},
\begin{equation}\label{eq_Adj=EV}
\left(f\mathbf{I}_n-\mathcal{X}\right)\mathrm{p}\left(f\mathbf{I}_n,\mathcal{X};\mathrm{a}\right)=
\mathrm{p}\left(f\mathbf{I}_n,\mathcal{X};\mathrm{a}\right)\left(f\mathbf{I}_n-\mathcal{X}\right)=
\mathrm{a}\left(f\right)\mathbf{I}_n.
\end{equation}
\begin{corollary}\label{cor_adj}
Let $\mathcal{X}\in R^{n\times n}$, $f\in R$ and $\mathrm{a}\left(x\right)\in R\left[x\right]\setminus{0}$ s.t. $\mathrm{a}\left(f\mathbf{I}_n-\mathcal{X}\right)=\mathbf{0}$. Then 
$\mathcal{X}\mathrm{p}\left(f\mathbf{I}_n,f\mathbf{I}_n-\mathcal{X};\mathrm{a}\right)=
\mathrm{p}\left(f\mathbf{I}_n,f\mathbf{I}_n-\mathcal{X};\mathrm{a}\right)\mathcal{X}=
\mathrm{a}\left(f\right)\mathbf{I}_n$. 
\end{corollary}
With $\mathrm{\Delta}\left(x\right)\hspace{-1pt}=\hspace{-1pt}\det\left(x\mathbf{I}_n-\mathcal{X}\right)$ being the characteristic polynomial of $\mathcal{X}$, $B\left(\lambda\right)\hspace{-1pt}=\hspace{-1pt}\mathrm{p}\left(\lambda,\mathcal{X};\mathrm{\Delta}\right)$ is called the adjoint matrix of $\mathcal{X}$ in~\cite[p.~82ff]{G60TheMat}. For the classical adjugate (also called adjoint)~\cite{R05CA}, $\mathcal{X}^{\mathrm{adj}}\hspace{-1pt}=\hspace{-1pt}\left(-1\right)^{n-1}\mathrm{p}\left(0,\mathcal{X};\mathrm{\Delta}\right)$, we have the well known relation $\mathcal{X}\mathcal{X}^{\mathrm{adj}}\hspace{-1pt}=\hspace{-1pt}\mathcal{X}^{\mathrm{adj}}\mathcal{X}\hspace{-1pt}=\hspace{-1pt}\det\left(\mathcal{X}\right)\mathbf{I}_n$, which follows from~\eqref{eq_Adj=EV}.
\subsection{Schur-like Complements and Opponents}\label{subsec_SC}
\begin{definition}\label{def_Schur}
Let $\mathcal{M}\hspace{-1pt}=\hspace{-1pt}\bigl(\hspace{-1pt}\begin{smallmatrix}\mathcal{A}&\mathcal{B}\\\mathcal{C}&\mathcal{D}\end{smallmatrix}\hspace{-1pt}\bigr)\hspace{-1pt}\in\hspace{-1pt}R^{\left(n+s\right)\times\left(n+s\right)}$, $\mathcal{A}\hspace{-1pt}\in\hspace{-1pt}R^{n\times n}$, $\mathcal{B}\hspace{-1pt}\in\hspace{-1pt}R^{n\times s}$, $\mathcal{C}\hspace{-1pt}\in\hspace{-1pt}R^{s\times n}$, and $\mathcal{D}\hspace{-1pt}\in\hspace{-1pt}R^{s\times s}$. Let $\mathcal{P}\hspace{-1pt}\in\hspace{-1pt}R^{n\times n}$ and $\mathcal{R}\hspace{-1pt}=\hspace{-1pt}\bigl(\begin{smallmatrix}\mathcal{R}_{A}&\mathcal{R}_{B}\\\mathcal{R}_{C}&\mathcal{R}_{D}\end{smallmatrix}\bigr)\hspace{-1pt}\in\hspace{-1pt}R^{\left(n+s\right)\times\left(n+s\right)}$ be not all zero and obey $\mathcal{P}\mathcal{A}\hspace{-1pt}=\hspace{-1pt}\mathcal{A}\mathcal{P}=a\mathbf{I}_n$ with $a\in R$, and $\mathcal{R}\mathcal{M}\hspace{-1pt}=\hspace{-1pt}m\mathbf{I}_{n+s}$ with $m\in R$. 
Set
\begin{equation}
\mathcal{L}=\left(\begin{array}{cc}\mathbf{I}_n&\mathbf{0}\\-\mathcal{C}\mathcal{P}&a\mathbf{I}_s\end{array}\right),\ \ 
\mathcal{U}=\left(\begin{array}{cc}\mathbf{I}_n&-\mathcal{P}\mathcal{B}\\\mathbf{0}&a\mathbf{I}_s\end{array}\right),
\ \text{and}\ \ \ 
\mathcal{S}=a\mathcal{D}-\mathcal{C}\mathcal{P}\mathcal{B}\in R^{s\times s}.
\end{equation}
$\mathcal{L}$ ( $\mathcal{U}$) is called a \emph{left} (\emph{right}) \emph{Schur-like multiplier},  $\mathcal{S}$ a \emph{Schur-like complement} of $\mathcal{A}$ in $\mathcal{M}$, and $\mathcal{R}_{D}\in R^{s\times s}$ an \emph{opponent} of $\mathcal{A}$ in $\mathcal{M}$.
\end{definition}
How Schur-like multipliers induce Schur-like complements is summarized in the following relations, which are easily verified and well known for $R=\mathbb{C}$ and $a=1$.
\begin{equation}\label{eq_LM_MU_LMU}
(i)\ \mathcal{L}\mathcal{M}\hspace{-2pt}=\hspace{-2pt}
\left(\begin{array}{cc}\mathcal{A}&\mathcal{B}\\\mathbf{0}&\mathcal{S}\end{array}\right),\ 
(ii)\ \mathcal{M}\mathcal{U}\hspace{-2pt}=\hspace{-2pt}
\left(\begin{array}{cc}\mathcal{A}&\mathbf{0}\\\mathcal{C}&\mathcal{S}\end{array}\right),\ 
(iii)\ \mathcal{L}\mathcal{M}\mathcal{U}\hspace{-2pt}=\hspace{-2pt}
\left(\begin{array}{cc}\mathcal{A}&\mathbf{0}\\\mathbf{0}&a\mathcal{S}\end{array}\right)
\end{equation}
Complements and opponents are related through
\begin{proposition}
In the notation of definition \ref{def_Schur}, 
$\mathcal{R}_{D}\mathcal{S}\hspace{-2pt}=\hspace{-2pt}am\mathbf{I}_s$ and
\begin{equation}\label{eq_a2RM}
a^2\mathcal{R}\mathcal{M}\hspace{-2pt}=\hspace{-2pt}\left(\begin{array}{cc}
am\mathcal{P}+\mathcal{P}\mathcal{B}\mathcal{R}_{\hspace{-2pt}D}\mathcal{C}\mathcal{P}&-a\mathcal{P}\mathcal{B}\mathcal{R}_{\hspace{-2pt}D}\\
-a\mathcal{R}_{\hspace{-2pt}D}\mathcal{C}\mathcal{P}&a^2\mathcal{R}_{\hspace{-2pt}D}
\end{array}\right)\mathcal{M}.
\end{equation}
\end{proposition}
\begin{proof}The first claim follows from the lower right corner of
\begin{equation*}
\left(\begin{array}{cc}m\mathbf{I}_n&-m\mathcal{P}\mathcal{B}\\\mathbf{0}&ma\mathbf{I}_s\end{array}\right)=
\mathcal{R}\mathcal{M}\mathcal{U}
=\mathcal{R}\left(\begin{array}{cc}\mathcal{A}&\mathbf{0}\\\mathcal{C}&\mathcal{S}\end{array}\right)
=\left(\begin{array}{cc}m\mathbf{I}_n&\mathcal{R}_{B}\mathcal{S}\\\mathbf{0}&\mathcal{R}_{D}\mathcal{S}\end{array}\right).
\end{equation*}
Using it and the definition of $\mathcal{S}$, the second claim can be directly verified.
\end{proof}
Sections~3 to~6 exploit the following lemma.
\begin{lemma}\label{lem_Schur}
In the notation of definition \ref{def_Schur},
\begin{align}
a^s\det\left(\mathcal{M}\right)&=\det\left(\mathcal{A}\right)\det\left(\mathcal{S}\right)\label{lem_Schur1}\\
\det\left(\mathcal{R}_{D}\right)\det\left(\mathcal{M}\right)&=\det\left(\mathcal{A}\right)m^s.\label{lem_Schur2}
\end{align}
\end{lemma}
\begin{proof}
\eqref{lem_Schur1} follows from \eqref{eq_LM_MU_LMU} yielding, e.g., $\det\left(\mathcal{L}\right)\det\left(\mathcal{M}\right)=\det\left(\mathcal{A}\right)\det\left(\mathcal{S}\right)$.
\newline
Assertion \eqref{lem_Schur2} follows from $\mathcal{R}\mathcal{M}=m\mathbf{I}_{n+s}$ implying
\begin{equation*}\det\left(\begin{array}{cc}\mathbf{I}_n&\mathbf{0}\\\mathcal{R}_{C}&\mathcal{R}_{D}\end{array}\right)
\det\left(\begin{array}{cc}\mathcal{A}&\mathcal{B}\\\mathcal{C}&\mathcal{D}\end{array}\right)
=\det\left(\begin{array}{cc}\mathcal{A}&\mathcal{B}\\\mathbf{0}&m\mathbf{I}_s\end{array}\right).\qedhere
\end{equation*}
\end{proof}
\subsection{Quotient Property}\label{subsec_QM}
In order to understand the structure of a sequential application of Schur-like multipliers, we consider a $3\times3$ block matrix over $R$.
Let 
\begin{equation}
\mathcal{M}'=\left(\begin{array}{ccc}
\mathcal{A}&\mathcal{B}&\mathcal{B}'\\
\mathcal{C}&\mathcal{D}&\mathcal{B}''\\
\mathcal{C}'&\mathcal{C}''&\mathcal{D}'
\end{array}\right)\in R^{\left(n+s+t\right)\times\left(n+s+t\right)},
\end{equation}
and let 
$\mathcal{P}\in R^{n\times n}$, $a\in R$, $\mathcal{S}\in R^{s\times s}$, $\mathcal{Q}\in R^{s\times s}$, $q\in R$ and $\mathcal{F}\in R^{s\times t}$ s.t.
\begin{equation}
\mathcal{P}\mathcal{A}=a\mathbf{I}_{n},\quad\mathcal{S}=a\mathcal{D}-\mathcal{C}\mathcal{P}\mathcal{B},\quad
\mathcal{Q}\mathcal{S}=q\mathbf{I}_{s}\quad\text{and}\quad\mathcal{F}=a\mathcal{C}''-\mathcal{C}'\mathcal{P}\mathcal{B}.
\end{equation}
If $\mathcal{R}
\in R^{\left(n+s\right)\times\left(n+s\right)}$ and $m\in R$ are set as
\begin{equation}
\mathcal{R}=\left(\begin{array}{cc}
q\mathcal{P}+\mathcal{P}\mathcal{B}\mathcal{Q}\mathcal{C}\mathcal{P}&-a\mathcal{P}\mathcal{B}\mathcal{Q}\\
-a\mathcal{Q}\mathcal{C}\mathcal{P}&a^2\mathcal{Q}
\end{array}\right)
\qquad\text{and}\qquad m=aq,
\end{equation}
then one shows directly that
$\mathcal{R}\left(\begin{smallmatrix}
\mathcal{A}&\mathcal{B}\\
\mathcal{C}&\mathcal{D}
\end{smallmatrix}\right)=m\mathbf{I}_{n+s}$
and, with suitable $\mathcal{E}$ and $\mathcal{S}'$,
\begin{align}
\left(\begin{array}{c|cc}
\mathcal{A}&\mathcal{B}&\mathcal{B}'\\\hline
\mathbf{0}&\mathcal{S}&\mathcal{E}\\
\mathbf{0}&\mathbf{0}&\mathcal{S}'
\end{array}\right)
&=
\left(\begin{array}{c|cc}
\mathbf{I}_{n}&\mathbf{0}&\mathbf{0}\\\hline
\mathbf{0}&\mathbf{I}_{s}&\mathbf{0}\\
\mathbf{0}&-\mathcal{F}\mathcal{Q}&q\mathbf{I}_{t}
\end{array}\right)\hspace{-1pt}
\left(\begin{array}{c|cc}
\mathbf{I}_{n}&\mathbf{0}\\\hline
-\left(\begin{array}{c}\mathcal{C}\\\mathcal{C}'\end{array}\right)\mathcal{P}
&a\mathbf{I}_{s+t}
\end{array}\right)\mathcal{M}'\label{eq_QP1}\\
\left(\begin{array}{cc|c}
\mathcal{A}&\mathcal{B}&\mathcal{B}'\\
\mathbf{0}&\mathcal{S}&\mathcal{E}\\\hline
\mathbf{0}&\mathbf{0}&\mathcal{S}'
\end{array}\right)
&=\left(\begin{array}{cc|c}
\mathbf{I}_{n}&\mathbf{0}&\mathbf{0}\\
-\mathcal{C}\mathcal{P}&a\mathbf{I}_{s}&\mathbf{0}\\\hline
\mathbf{0}&\mathbf{0}&\mathbf{I}_{t}
\end{array}\right)
\left(\begin{array}{c|c}
\mathbf{I}_{n+s}&\mathbf{0}\\\hline
-\left(\begin{array}{cc}\mathcal{C}'&\mathcal{C}''\end{array}\right)\mathcal{R}&m\mathbf{I}_{t}
\end{array}\hspace{-1pt}\right)\hspace{-1pt}\mathcal{M}'.\label{eq_QP2}
\end{align}
If only $\mathcal{R}\hspace{-2pt}=\hspace{-2pt}\left(\begin{smallmatrix}
\mathcal{R}_A&\mathcal{R}_B\\
\mathcal{R}_C&\mathcal{R}_D
\end{smallmatrix}\right)\hspace{-2pt}\in\hspace{-2pt}R^{\left(n+s\right)\times\left(n+s\right)}$ with $\mathcal{R}\left(\begin{smallmatrix}
\mathcal{A}&\mathcal{B}\\
\mathcal{C}&\mathcal{D}
\end{smallmatrix}\right)\hspace{-2pt}=\hspace{-2pt}m\mathbf{I}$, from~\eqref{eq_QP2} follows
\begin{equation}\label{eq_QP3}
\hspace{-5pt}\left(\hspace{-1pt}\begin{array}{c|cc}
\mathcal{A}&\mathcal{B}&\mathcal{B}'\\\hline
\mathbf{0}&\mathcal{S}&\mathcal{E}\\
\mathbf{0}&\mathbf{0}&a^2\mathcal{S}'
\end{array}\hspace{-1pt}\right)
\hspace{-3pt}=\hspace{-3pt}
\left(\hspace{-1pt}\begin{array}{c|cc}
\mathbf{I}_{n}&\mathbf{0}&\mathbf{0}\\\hline
\mathbf{0}&\mathbf{I}_{s}&\mathbf{0}\\
\mathbf{0}&-\mathcal{F}\mathcal{R}_D&\hspace{-2pt}am\mathbf{I}_{t}
\end{array}\hspace{-1pt}\right)
\hspace{-6pt}
\left(\hspace{-1pt}\begin{array}{c|cc}
\mathbf{I}_{n}&\mathbf{0}\\\hline
-\left(\begin{array}{c}\mathcal{C}\\\mathcal{C}'\end{array}\right)\mathcal{P}
&a\mathbf{I}_{s+t}
\end{array}\hspace{-1pt}\right)
\hspace{-3pt}
\mathcal{M}'
\end{equation}
using~\eqref{eq_a2RM}. 
For $R=\mathbb{C}$ and $a=q=m=1$ those relations imply the famous quotient property of Schur complements~\cite{Crabtree69, Ostrowski71}. From \eqref{eq_QP2} and \eqref{eq_QP3} follows
\begin{proposition}\label{pro_QP}
Let $\mathcal{A}_{k}$ be the $k$-th leading principal submatrix of $\mathcal{M}\in R^{N\times N}$ and let $\mathcal{S}_{\tilde{n}}$ be a Schur-like complement of $\mathcal{A}_{\tilde{n}}$ in $\mathcal{M}$. Set $\mathcal{S}^{(0)}=\mathcal{M}$ and let $\mathcal{S}^{(k+1)}$ be a Schur-like complement of $\mathcal{S}^{(k)}_{11}$ in $\mathcal{S}^{(k)}$. If $\mathcal{A}_{k}$ is regular for all $k\leq\tilde{n}$, then there is a regular $\varsigma\in R$ s.t. $\mathcal{S}^{\tilde{n}}=\varsigma\mathcal{S}_{\tilde{n}}$. 
\end{proposition}
\begin{proof}
The case $\tilde{n}\hspace{-1pt}=\hspace{-1pt}1$ is obvious. Let $\varsigma'\in R$ be regular s.t. $\mathcal{S}^{(\tilde{n}-1)}\hspace{-1pt}=\hspace{-1pt}\varsigma'\mathcal{S}_{\tilde{n}-1}$.
Since $\mathcal{A}_{\tilde{n}-1}$ is regular, there is regular $p$ s.t. $p^2\mathcal{S}_{\tilde{n}}$ is a Schur-like complement of $\left(\mathcal{S}_{\tilde{n}-1}\right)_{11}$ in $\mathcal{S}_{\tilde{n}-1}$ according to ~\eqref{eq_QP2} and \eqref{eq_QP3} with $\left(n,s,t\right)\hspace{-1pt}=\hspace{-1pt}\left(\tilde{n}-1,1,N-\tilde{n}\right)$. Hence, $\mathcal{S}^{\tilde{n}}=p^2\varsigma'\mathcal{S}_{\tilde{n}}$ with $p^2\varsigma'$ regular. Proposition~\ref{pro_QP} follows by induction.
\end{proof}
\subsection{Strict (Double) Diagonal Dominance}\label{subsec_SDD}
Let $R$ be equipped with a multiplicative norm, i.e. an absolute value, which is a function 
$\lvert\ \rvert:R\rightarrow\mathbb{R}_{\geq0}\quad\text{s.t.}$
$$(i)\ \lvert r\rvert=0\Leftrightarrow r=0,\quad
(ii)\ \lvert r+r'\rvert\leq\lvert r\rvert+\lvert r'\rvert,\quad
(iii)\ \lvert rr'\rvert=\lvert r\rvert\lvert r'\rvert.$$
It is well known that condition (iii) ensures that $\lvert-r\rvert=\lvert r\rvert$ and, together with (i), forces $R$ to have no zero divisor.
\begin{definition}\label{def_sd}
$\mathcal{X}\in R^{N\times N}$ is called \emph{strictly diagonally dominant (sd)} if
\begin{equation}
\forall\,i\in\left\{1,\ldots,N\right\}\ \lvert\mathcal{X}_{ii}\rvert>\textstyle{\text{$\sum_{k\ne i}$}}\lvert\mathcal{X}_{ik}\rvert,
\end{equation}
and $\mathcal{X}$ is called \emph{strictly doubly diagonally dominant (sdd)} if
\begin{equation}
\forall\,i,j\in\left\{1,\ldots,N\right\}\ \lvert\mathcal{X}_{ii}\mathcal{X}_{jj}\rvert>\bigl(\textstyle{\text{$\sum_{k\ne i}$}}\lvert\mathcal{X}_{ik}\rvert\bigr)\bigl(\textstyle{\text{$\sum_{l\ne j}$}}\lvert\mathcal{X}_{jl}\rvert\bigr).
\end{equation}
\end{definition}
The following properties of sd and sdd matrices are easy to see.
\begin{proposition}\label{pro_sdd}
\begin{itemize}
\item[(i)] Every sd matrix of size $N\geq2$ is sdd.
\item[(ii)] 
If $\mathcal{X}$ is sd (sdd), then every principal submatrix of size $n\geq2$ is sd (sdd) and the diagonal entries are regular.
\item[(iii)]
If $\mathcal{X}$ is sdd but not sd, then there is exactly on index $i$ s.t.\newline
$\lvert\mathcal{X}_{ii}\rvert\leq\sum_{k\neq i}\lvert\mathcal{X}_{ik}\rvert$, and $\forall\,j\neq i\ \lvert\mathcal{X}_{jj}\rvert>\lvert\sum_{l\neq j}\mathcal{X}_{jl}\rvert$.
\item[(iv)]If $\mathcal{X}\hspace{-1pt}\in\hspace{-1pt}R^{2\times2}$ is sdd and $0\neq\mathcal{P}\hspace{-1pt}\in\hspace{-1pt}R$, then $\lvert\left(\mathcal{P}\mathcal{X}_{11}\right)\mathcal{X}_{22}-\mathcal{X}_{21}\mathcal{P}\mathcal{X}_{12}\rvert>0$.
\end{itemize}
\end{proposition}
We will use the following lemma, which is familiar from $R=\mathbb{C}$~\cite{LITsa97}.
\begin{lemma}\label{lem_sdd}
Let $\mathcal{M}=\left(\begin{smallmatrix}\mathcal{A}&\mathcal{B}\\\mathcal{C}&\mathcal{D}\end{smallmatrix}\right)\in R^{\left(n+s\right)\times\left(n+s\right)}$ and let $\mathcal{P}\in R^{n\times n}$ not all zero and $a\in R$ s.t. $\mathcal{P}\mathcal{A}=a\mathbf{I}_n$. If $\mathcal{M}$ is sd (sdd), then
\begin{itemize}
\item[(i)] the Schur-like complement $\mathcal{S}=\mathrm{a}\mathcal{D}-\mathcal{C}\mathcal{P}\mathcal{B}$ is sd (sdd), and
\item[(ii)] $\mathcal{M}$, $\mathcal{A}$ and $\mathcal{S}$ are regular.
\end{itemize}
\end{lemma}
\begin{proof}
By proposition~\ref{pro_sdd}(iv), the lemma holds for $n+s=2$. In order to use induction, we first consider the case $n=1$ and abbreviate
\begin{equation}
b_{\gamma}=\lvert\mathcal{B}_{1\gamma}\rvert,\quad
b=\textstyle{\text{$\sum_{\gamma}b_{\gamma}$}},\quad
c_{\gamma}=\lvert\mathcal{C}_{\gamma1}\rvert,\quad
d_{\gamma}=\textstyle{\text{$\sum_{\gamma'\neq\gamma}$}}\lvert\mathcal{D}_{\gamma\gamma'}\rvert.
\end{equation}
According to proposition~\ref{pro_sdd}(iii), we distinguish three possibilities.
$$(1)\ \mathcal{M}\text{ is sd},\qquad(2)\ b\geq\lvert\mathcal{A}\rvert>0,\qquad(3)\ \exists\,\alpha\text{ s.t. }d_{\alpha}+c_{\alpha}\geq\lvert\mathcal{D}_{\alpha\alpha}\rvert>0$$
(1) $\mathcal{M}$ is sd.
\begin{align}
\lvert\mathcal{S}_{\alpha\alpha}\rvert&=\lvert a\mathcal{D}_{\alpha\alpha}-\mathcal{C}_{\alpha1}\mathcal{P}\mathcal{B}_{1\alpha}\rvert
\geq\lvert\mathcal{P}\rvert\left(\lvert A\mathcal{D}_{\alpha\alpha}\rvert-c_{\alpha}b_{\alpha}\right)\\
&>\lvert\mathcal{P}\rvert\left(\lvert\mathcal{A}\rvert\left(d_{\alpha}+c_{\alpha}\right)-c_{\alpha}b_{\alpha}\right)
\quad\qquad\qquad\qquad\qquad
\text{\small{(since $\mathcal{M}$ is sd)}}\\
&\geq\lvert\mathcal{P}\rvert\left(\lvert\mathcal{A}\rvert d_{\alpha}+
c_{\alpha}\textstyle{\text{$\sum_{\gamma\neq\alpha}$}}b_{\gamma}\right)
\quad\qquad\qquad\qquad\qquad\,
\text{\small{(since $\lvert\mathcal{A}\rvert> b$)}}\\
&\geq\textstyle{\text{$\sum_{\gamma\neq\alpha}$}}\lvert
a\mathcal{D}_{\alpha\beta}-\mathcal{C}_{\alpha1}\mathcal{P}\mathcal{B}_{1\gamma}
\rvert=\textstyle{\text{$\sum_{\gamma\neq\alpha}$}}\lvert\mathcal{S}_{\alpha\gamma}\rvert
\end{align}
(2) $b\geq\lvert\mathcal{A}\rvert>0$.
\begin{align}
\lvert\mathcal{S}_{\alpha\alpha}\rvert
&=\lvert a\mathcal{D}_{\alpha\alpha}
-\mathcal{C}_{\alpha1}\mathcal{P}\mathcal{B}_{1\alpha}\rvert
\geq\lvert\mathcal{P}\rvert\left(\lvert \mathcal{A}\mathcal{D}_{\alpha\alpha}\rvert
-c_{\alpha}b_{\alpha}\rvert\right)\\
&>\lvert\mathcal{P}\rvert\left(b\left(d_{\alpha}+c_{\alpha}\right)-c_{\alpha}b_{\alpha}\right)
\quad\qquad\qquad\qquad\qquad\quad
\text{\small{(since $\mathcal{M}$ is sdd)}}\\
&\geq\lvert\mathcal{P}\rvert\left(\lvert\mathcal{A}\rvert d_{\alpha}+
c_{\alpha}\textstyle{\text{$\sum_{\gamma\neq\alpha}$}}b_{\gamma}\right)
\quad\qquad\qquad\qquad\qquad\ 
\text{\small{(since $b\geq\lvert\mathcal{A}\rvert$)}}\\
&\geq\textstyle{\text{$\sum_{\gamma\neq\alpha}$}}\lvert a\mathcal{D}_{\alpha\gamma}-\mathcal{C}_{\alpha1}\mathcal{P}\mathcal{B}_{1\gamma}\rvert
=\sum_{\gamma\neq\alpha}\vert\mathcal{S}_{\alpha\gamma}\rvert
\end{align}
\phantom{(3) }
\newline
(3) $d_{\alpha}+c_{\alpha}\geq\lvert\mathcal{D}_{\alpha\alpha}\rvert>0$. The case of $s=1$ is considered in proposition~\ref{pro_sdd}(iv). 
\phantom{(3) }Let $s\geq2$ and set $y=\lvert\mathcal{D}_{\alpha\alpha}\rvert/\left(d_{\alpha}+c_{\alpha}\right)\leq1$. Since $\mathcal{M}$ is sdd, we obtain 
\begin{equation}\label{eq_yA>b}
\lvert\mathcal{A}\rvert\geq y\lvert\mathcal{A}\rvert>b,
\qquad\forall\beta\neq\alpha\ 
\lvert\mathcal{D}_{\beta\beta}\rvert\geq y\lvert\mathcal{D}_{\beta\beta}\rvert>d_{\beta}+c_{\beta}
\qquad\text{and}
\end{equation}
\begin{align}
\lvert\mathcal{S}_{\alpha\alpha}\mathcal{S}_{\beta\beta}\rvert
&=\lvert\left(
a\mathcal{D}_{\alpha\alpha}-\mathcal{C}_{\alpha1}\mathcal{P}\mathcal{B}_{1\alpha}\right)\left(
a\mathcal{D}_{\beta\beta}-\mathcal{C}_{\beta1}\mathcal{P}\mathcal{B}_{1\beta}
\right)\rvert\\
&\geq\lvert\mathcal{P}\rvert^2\left(\lvert
\mathcal{A}\mathcal{D}_{\alpha\alpha}\rvert-c_{\alpha}b_{\alpha}\right)
\left(\lvert\mathcal{A}\mathcal{D}_{\beta\beta}\rvert-c_{\beta}b_{\beta}
\right)\\
&=\lvert\mathcal{P}\rvert^2
\left(\lvert\mathcal{A}\rvert^2d_{\alpha}y\lvert\mathcal{D}_{\beta\beta}\rvert
-\lvert\mathcal{A}\rvert yd_{\alpha}c_{\beta}b_{\beta}\right.\nonumber\\
&\phantom{>\lvert\mathcal{P}\rvert^2(}
\left.+c_{\alpha}\left(\lvert\mathcal{A}\rvert y-b_{\alpha}\right)
\left(\lvert\mathcal{A}\mathcal{D}_{\beta\beta}\rvert-c_{\beta}b_{\beta}\right)
\right)\\
&>\lvert\mathcal{P}\rvert^2
\left(\lvert\mathcal{A}\rvert^2d_{\alpha}
\left(d_{\beta}+c_{\beta}\right)
-\lvert\mathcal{A}\rvert d_{\alpha}c_{\beta}b_{\beta}\right.\nonumber\\
&\phantom{>\lvert\mathcal{P}\rvert^2(}
\left.+c_{\alpha}\left(b-b_{\alpha}\right)
\left(\lvert\mathcal{A}\rvert\left(d_{\beta}+c_{\beta}\right)-c_{\beta}b_{\beta}\right)\right)
\qquad\qquad\ \text{\small{(by \eqref{eq_yA>b})}}\\
&=\lvert\mathcal{P}\rvert^2
\left(\lvert\mathcal{A}\rvert d_{\alpha}+c_{\alpha}\left(b-b_{\alpha}\right)\right)
\left(\lvert\mathcal{A}\rvert d_{\beta}+c_{\beta}\left(\lvert\mathcal{A}\rvert-b_{\beta}\right)\right)\nonumber\\
&\geq
\left(\lvert a\rvert d_{\alpha}+
c_{\alpha}\lvert\mathcal{P}\rvert\textstyle{\text{$\sum_{\gamma\neq\alpha}$}}b_{\gamma}\right)
\left(\lvert a\rvert d_{\beta}+
c_{\beta}\lvert\mathcal{P}\rvert\textstyle{\text{$\sum_{\delta\neq\beta}$}}b_{\delta}\right)
\quad\text{\small{(by \eqref{eq_yA>b})}}\\
&\geq
\textstyle{\text{$\sum_{\gamma\neq\alpha}$}}
\lvert a\mathcal{D}_{\alpha\gamma}-\mathcal{C}_{\alpha1}\mathcal{P}\mathcal{B}_{1\gamma}\rvert
\textstyle{\text{$\sum_{\delta\neq\beta}$}}
\lvert a\mathcal{D}_{\beta\delta}-\mathcal{C}_{\beta1}\mathcal{P}\mathcal{B}_{1\delta}\rvert
\\&
=
\textstyle{\text{$\sum_{\gamma\neq\alpha}$}}\lvert\mathcal{S}_{\alpha\gamma}\rvert\textstyle{\text{$\sum_{\delta\neq\beta}$}}\lvert\mathcal{S}_{\beta\delta}\rvert
\end{align}
Under the additional assumption that all proper principal submatrices of $\mathcal{M}$ are regular, which implies, by proposition~\ref{pro_reg}, that $\mathcal{A}$, $\mathcal{P}$ and $a$ are regular, we can, according to proposition~\ref{pro_QP}, apply the above result recursively, showing (i) for all $n\geq1$ under the made assumption.

In order to show that the assumption holds and to prove (ii), we consider any principal submatrix of $\mathcal{M}$ of the form
$\mathcal{M}'=\bigl(\begin{smallmatrix}\mathcal{A}'&\mathcal{B}'\\\mathcal{C}'&d'\end{smallmatrix}\bigr)\in R^{\left(n'+1\right)\times\left(n'+1\right)}$ s.t. $\mathcal{A}'\in R^{n'\times n'}$ and all its principal submatrices are regular, which holds for $n'=1$. By proposition~\ref{pro_sdd}(ii) and~\ref{pro_sdd}(i), $\mathcal{M}'$ is sdd. As shown above a Schur-like complement of $\mathcal{A}'$ in $\mathcal{M}'$, since it is of size $1$, is regular. Thus, $\det\left(\mathcal{M}'\right)\hspace{-2pt}\neq\hspace{-2pt}0$ by lemma~\ref{lem_Schur}.  
By induction over $n'$ follows that $\mathcal{M}$ and all its principal submatrices are regular, which completes the proof of (i) and implies that $\mathcal{A}$ and, by lemma~\ref{lem_Schur}, $\mathcal{S}$ are regular.
\end{proof}
\section{Basic Theorems}\label{sec_BT}
Let R be a commutative ring and let $\mathcal{M}\in R^{N\times N}$ be partitioned as \begin{equation}\mathcal{M}=\left(\begin{array}{cc}\mathcal{A}&\mathcal{B}\\\mathcal{C}&\mathcal{D}\end{array}\right)\in R^{\left(n+s\right)\times\left(n+s\right)}.
\end{equation}
\begin{theorem}\label{theo_main2} Let $\mathcal{X},\mathcal{Y}\hspace{-1pt}\in\hspace{-1pt}R^{n\times n}$ 
be regular, or let any pair in $\left\{\mathcal{A},\mathcal{X},\mathcal{Y}\right\}$ commute. Let $f\hspace{-1pt}\in\hspace{-1pt}R$ and set 
$\tilde{\mathcal{A}}\hspace{-1pt}=\hspace{-1pt}f\mathbf{I}_n-\mathcal{X}\mathcal{A}\mathcal{Y}$. Let $\mathrm{a}\left(x\right)\hspace{-1pt}\in\hspace{-1pt}R\left[x\right]\hspace{-1pt}\setminus\hspace{-1pt}\left\{0\right\}$ be a (non trivial) annihilating polynomial for $\tilde{\mathcal{A}}$ and set
\begin{equation}
\mathcal{P}=\mathcal{Y}\mathrm{p}\left(f,\tilde{\mathcal{A}};\mathrm{a}\right)\mathcal{X}
\qquad\text{and}\qquad
\mathcal{S}=\mathrm{a}\left(f\right)\mathcal{D}-\mathcal{C}\mathcal{P}\mathcal{B}.
\end{equation}
Then
$\left(\mathrm{a}\left(f\right)\right)^s\det\left(\mathcal{M}\right)=\det\left(\mathcal{A}\right)\det\left(\mathcal{S}\right).$
\end{theorem}
\begin{proof}From corollary~\ref{cor_adj} follows $\mathbf{0}=\mathcal{Y}\left(\mathrm{p}\mathcal{X}\mathcal{A}\mathcal{Y}-\mathrm{a}\left(f\right)\right)=\left(\mathcal{P}\mathcal{A}-\mathrm{a}\left(f\right)\right)\mathcal{Y}$ 
and similarly $\mathbf{0}\hspace{-1pt}=\hspace{-1pt}\mathcal{X}\left(\mathcal{A}\mathcal{P}-\mathrm{a}\left(f\right)\right)$. For regular $\mathcal{X}$, $\mathcal{Y}$ follows $\mathcal{P}\hspace{-1pt}\mathcal{A}\hspace{-1pt}=\hspace{-1pt}\mathcal{A}\mathcal{P}\hspace{-1pt}=\hspace{-1pt}\mathrm{a}\left(f\right)\mathbf{I}_n$, which also follows from corollary~\ref{cor_adj} for pairwise commuting $\mathcal{A},\mathcal{X}$ and $\mathcal{Y}$. Thus, lemma~\ref{lem_Schur} applies.
\end{proof}
The next lemma relates the kernels of the homomorphisms given by $\mathcal{M}$ and $\mathcal{S}$. Let $L$ be a left $R$-module and denote the \emph{kernel} of $\mathcal{Z}\in R^{q\times r}$ as
\begin{equation}
\ker\left(\mathcal{Z}\right)=\left\{\mathbf{y}\in L^{r}:\mathcal{Z}\mathbf{y}=0_L\,\mathbf{j}_q\right\}
\quad\text{with}\quad\mathbf{j}_q=\left(1,\ldots,1\right)^T=\left\{1\right\}^q.
\end{equation}
\begin{lemma}\label{lem_ker}
Maintaining the notation of theorem \ref{theo_main2}, let 
$\mathbf{v}_{0}\in\ker\left(\mathcal{A}\right)\cap\ker\left(\mathcal{C}\right)$ and $\mathbf{w}_{\mathcal{S}}\in\ker\left(\mathcal{S}\right)$. Then
\begin{equation}
\mathbf{u}=
\left(\begin{array}{c}
-\mathcal{Y}\mathrm{p}\left(f,\tilde{\mathcal{A}};\mathrm{a}\right)\mathcal{X}\mathcal{B}\\
\mathrm{a}\left(f\right)\mathbf{I}_s
\end{array}\right)\mathbf{w}_{\mathcal{S}}+
\left(\begin{array}{c}
\mathbf{v}_{0}\\
0_L\,\mathbf{j}_s
\end{array}\right)
\in\ker\left(\mathcal{M}\right).
\end{equation}
\end{lemma}
\begin{proof} According to the proof of theorem \ref{theo_main2}, $\mathcal{A}\mathcal{P}\hspace{-1pt}=\hspace{-1pt}\mathrm{a}\left(f\right)\mathbf{I}_n$. Therefore,
\begin{equation*}
\left(\begin{array}{cc}
\mathcal{A}&\mathcal{B}\\
\mathcal{C}&\mathcal{D}
\end{array}\right)
\mathbf{u}
=
\left(\begin{array}{c}
\left(-\mathcal{A}\mathcal{P}+\mathrm{a}\left(f\right)\right)\mathcal{B}\\
\mathcal{S}
\end{array}\right)\mathbf{w}_{\mathcal{S}}+
\left(\begin{array}{c}
\mathbf{A}\\
\mathbf{C}
\end{array}\right)\mathbf{v}_{0}=
0_L\mathbf{j}_{n+s}.
\qedhere\end{equation*}
\end{proof}
The concluding result of this section is complementary to theorem \ref{theo_main2}. It provides a relation between the determinants of $\mathcal{M}$ and $\mathcal{A}$ based on an annihilating polynomial for a scaled and shifted $\mathcal{M}$.
\begin{theorem}\label{theo_conv2}
Let $f\hspace{-1pt}\in\hspace{-1pt}R$, $\mathcal{X}\hspace{-1pt}\in\hspace{-1pt}R^{N\times N}$ and set $\tilde{\mathcal{M}}\hspace{-1pt}=\hspace{-1pt}f\mathbf{I}_{N}-\mathcal{X}\mathcal{M}$.
Let $\mathrm{m}\bigl(x\bigr)\hspace{-1pt}\in\hspace{-1pt}R\left[x\right]$ be a non trivial annihilating polynomial for $\tilde{\mathcal{M}}$, let
\begin{equation}
\mathrm{p}\bigl(x,\tilde{\mathcal{M}};\mathrm{m}\bigr)\mathcal{X}=\left(\begin{array}{cc}\mathcal{R}_{A}\left(x\right)&\mathcal{R}_{B}\left(x\right)\\\mathcal{R}_{C}\left(x\right)&\mathcal{R}_{D}\left(x\right)\end{array}\right)
\in R^{\left(n+s\right)\times\left(n+s\right)}
\end{equation}
be partitioned conformally to $\mathcal{M}$ and abbreviate
$\mathcal{K}=\mathcal{R}_{D}\left(f\right)\in R^{s\times s}$.
Then 
\begin{equation*}
\left(\mathrm{m}\left(f\right)\right)^s\det\left(\mathcal{A}\right)=\det\left(\mathcal{M}\right)\det\left(\mathcal{K}\right).
\end{equation*}
\end{theorem}
\begin{proof}
Since $\mathrm{p}\bigl(f,\tilde{\mathcal{M}};\mathrm{m}\bigr)\mathcal{X}\mathcal{M}=\mathrm{m}\left(f\right)\mathbf{I}_N$ by corollary~\ref{cor_adj}, lemma~\ref{lem_Schur} applies.
\end{proof}
\section{Isospectral Size Reduction}\label{sec_ISR}
Let $F$ be the field of meromorphic functions on the complex domain $D$ and let $\mathcal{M}=\bigl(\begin{smallmatrix}\mathcal{A}&\mathcal{B}\\\mathcal{C}&\mathcal{D}\end{smallmatrix}\bigr)\in F^{\left(n+s\right)\times\left(n+s\right)}$.
Let $\mathcal{X},\mathcal{Y}\in F^{n\times n}$ be regular, $f\in F$ and let $\mathrm{a}\left(x\right)\in F\left[x\right]\setminus\left\{0\right\}$ be a polynomial s.t. 
$\mathrm{a}\left(f\mathbf{I}_n-\mathcal{X}\mathcal{A}\mathcal{Y}\right)\equiv0$.
Let $P\hspace{-1pt}\subset\hspace{-1pt}D$ be the set of isolated poles at which $f$, an entry of $\mathcal{M}$, $\mathcal{X}$ or $\mathcal{Y}$, or a coefficient of $\mathrm{a}\left(x\right)$ is not defined. Denote $D\hspace{-1pt}\setminus\hspace{-1pt}P$ by $D_P$. We call $\lambda\in D_P$ and non vanishing $\mathbf{u}=\bigl(\hspace{-1pt}\begin{smallmatrix}\mathbf{v}\\\mathbf{w}\end{smallmatrix}\hspace{-1pt}\bigr)\hspace{-1pt}\in\hspace{-1pt}\mathbb{C}^{\left(n+s\right)}$ s.t.
\begin{equation}\label{eq_0=Mu}
\mathbf{0}=\mathcal{M}\left(\lambda\right)\mathbf{u}=\left(\begin{array}{cc}\mathcal{A}\left(\lambda\right)&\mathcal{B}\left(\lambda\right)\\\mathcal{C}\left(\lambda\right)&\mathcal{D}\left(\lambda\right)\end{array}\right)
\left(\begin{array}{c}\mathbf{v}\\\mathbf{w}\end{array}\right)
,\text{ i.e. }0=\det\left(\mathcal{M}\left(\lambda\right)\right),
\end{equation}
\emph{latent root} and \emph{latent vector} of $\mathcal{M}$, respectively. 
We will show that the latent root problem in \eqref{eq_0=Mu} can basically be reduced to the latent root problem of a Schur-like complement in $F^{s\times s}$, and provide a relation for latent vectors, too.
Additionally, considering spectral estimates based on the theorems of Geshgorin and Brauer~\cite{BUNIMOVICH20121429}, we show that the estimates corresponding to the Schur-like complement improve those of the original problem.
In light of those results this reduction sheme generalizes the so called \emph{isospectral graph reduction} proposed in~\cite{BUNIMOVICH20121429}, which arises with the additional conditions
\begin{itemize}
\item[(i)] $\mathcal{A}$ is triangular (up to simultaneous row and column permutation),
\item[(ii)] $\mathcal{X}$ and $\mathcal{X}$ are diagonal with
$\mathcal{X}_{ii}\mathcal{Y}_{ii}=\mathcal{A}_{ii}^{-1}$,
\item[(iii)] $f\equiv1$ and $\mathrm{a}\left(x\right)=x^{n}$.
\end{itemize}
\paragraph{Size Reduction}
We set $\mathcal{P}\hspace{-2pt}=\hspace{-2pt}\mathcal{Y}\mathrm{p}\hspace{-1pt}\left(f,f\mathbf{I}\hspace{-2pt}-\hspace{-2pt}\mathcal{X}\mathcal{A}\mathcal{Y};\mathrm{a}\right)\hspace{-1pt}\mathcal{X}$, using $\mathrm{p}\hspace{-1pt}\left(x,y;\mathrm{a}\right)$ given in definition~\ref{def_adj},
and consider the Schur-like complement
\begin{equation}\label{eq_isoS}
\mathcal{S}\left(\lambda\right)=\left(a\left(f\right)\right)\left(\lambda\right)\mathcal{D}\left(\lambda\right)-\mathcal{C}\left(\lambda\right)\mathcal{P}\left(\lambda\right)\mathcal{B}\left(\lambda\right).
\end{equation}
By theorem~\ref{theo_main2} we have
$\left(\mathrm{a}\left(f\right)\right)^{s}\det\left(\mathcal{M}\right)=\det\left(\mathcal{A}\right)\det\left(\mathcal{S}\right)$,
and from corollary~\ref{cor_adj} follows $\mathcal{P}\mathcal{A}=\left(\mathrm{a}\left(f\right)\right)\mathbf{I}_{n}$, i.e. $
\det\left(\mathcal{P}\right)\det\left(\mathcal{A}\right)=\left(\mathrm{a}\left(f\right)\right)^{n}$, as shown in the proof of theorem~\ref{theo_main2}. This establishes the following corollary
\begin{corollary}
\begin{itemize}
\item[(i)] Every latent root of $\mathcal{A}$ is a root of $\mathrm{a}\left(f\right)$.
\item[(ii)] Every latent root of $\mathcal{M}$ is a latent root of $\mathcal{S}$ or $\mathcal{A}$.
\item[(iii)] A latent root of $\mathcal{S}$ which is not a root of $\mathrm{a}\left(f\right)$ is a latent root of $\mathcal{M}$. 
\end{itemize}
\end{corollary}
By lemma~\ref{lem_ker} we have the following relation for latent vectors
\begin{corollary}If $\left(\lambda_0,\mathbf{w}_0\right)$ is a latent pair of $\mathcal{S}$, i.e. $\mathbf{0}=\mathcal{S}\left(\lambda_0\right)\mathbf{w}_0$, s.t. $\lambda_0$ is not a root of $a\left(f\right)$,
then $\left(\lambda_0,\left(\begin{array}{c}-\mathcal{P}\hspace{-1pt}\left(\lambda_0\right)\mathcal{B}\left(\lambda_0\right)\\
\mathrm{a}\left(f\right)\left(\lambda_0\right)\end{array}\right)\mathbf{w}_0\right)
$ is a latent pair of $\mathcal{M}$.
\end{corollary}
\paragraph{Improved Spectral Bounds}
It will be shown that $\mathcal{S}\left(\lambda\right)$ in \eqref{eq_isoS} admits an improvement of Geshgorin and Brauer like spectral approximations, which  generalizes corresponding results in~\cite[corollaries 2 and 3]{BUNIMOVICH20121429}.
We consider the theorems of Geshgorin~\cite{Gersch31} and Brauer~\cite{Brauer47} in the following form~\cite{BUNIMOVICH20121429}.
\begin{definition}
Let $\mathcal{X}$ be a square matrix over F. Define the (row based) \emph{Gershgorin region} and \emph{Brauer region} of $\mathcal{X}$, respectively, as
$$G\left(\mathcal{X}\right)=\bigcup_i
\left\{\lambda:\left|\mathcal{X}_{ii}\left(\lambda\right)\right|\leq\textstyle{\text{$\sum_{j\neq i}$}}\left|\mathcal{X}_{ij}\left(\lambda\right)\right|\right\}\qquad\text{and}$$
$$K\left(\mathcal{X}\right)=\bigcup_{i<k}
\left\{\lambda:\lvert\mathcal{X}_{ii}\left(\lambda\right)\mathcal{X}_{kk}\left(\lambda\right)\rvert\leq
\textstyle{\text{$\sum_{j\neq i}$}}\lvert\mathcal{X}_{ij}\left(\lambda\right)\rvert
\textstyle{\text{$\sum_{l\neq k}$}}\lvert\mathcal{X}_{kl}\left(\lambda\right)\rvert
\right\}$$
\end{definition}
\begin{lemma}\label{lem_LXsubKXsubGX}
$K\left(\mathcal{X}\right)\subseteq G\left(\mathcal{X}\right)$, and $0=\det\left(\mathcal{X}\left(\lambda_0\right)\right)$ implies
$\lambda_0\in K\left(\mathcal{X}\right)$.
\end{lemma}
\begin{proof}
Let $\lambda\in K\left(\mathcal{X}\right)$. Since $\mathcal{X}\left(\lambda\right)$ is not sdd, it is not sd, by proposition~\ref{pro_sdd}(i). Hence $\lambda\in G\left(\mathcal{X}\right)$.
If $\lambda_0$ is a latent root of $\mathcal{X}$, then $\mathcal{X}\left(\lambda_0\right)$ is not regular, hence by lemma~\ref{lem_sdd}, not sdd. Thus, $\lambda_0\in K\left(\mathcal{X}\right)$.
\end{proof}
We now consider the latent root problem in~\eqref{eq_0=Mu} together with~\eqref{eq_isoS}.
\begin{theorem}Let $\rho\hspace{-1pt}\left(f'\right)$ denote the set of all distinct roots of $f'\in\mathcal{F}$ in $D_P$.
\begin{itemize}
\item[(i)] $\rho\left(\det\left(\mathcal{M}\right)\right)\hspace{2pt}\subseteq\hspace{2pt}G\hspace{-1pt}\left(\mathcal{S}\right)\hspace{-1pt}\cup\hspace{-1pt}\rho\left(\det\left(\mathcal{A}\right)\right)\hspace{2pt}\subseteq\hspace{2pt}G\hspace{-1pt}\left(\mathcal{M}\right)\hspace{-1pt}\cup\hspace{-1pt}\rho\left(\mathrm{a}\left(f\right)\right)$ and 
\item[(ii)] $\rho\left(\det\left(\mathcal{M}\right)\right)\hspace{2pt}\subseteq\hspace{2pt}K\hspace{-1pt}\left(\mathcal{S}\right)\hspace{-1pt}\cup\hspace{-1pt}\rho\left(\det\left(\mathcal{A}\right)\right)\hspace{2pt}\subseteq\hspace{2pt}K\hspace{-1pt}\left(\mathcal{M}\right)\hspace{-1pt}\cup\hspace{-1pt}\rho\left(\mathrm{a}\left(f\right)\right)$.
\end{itemize}
\end{theorem}
\begin{proof}
$\rho\left(\det\left(\mathcal{M}\right)\right)\subseteq \rho\left(\det\left(\mathcal{S}\right)\right)\cup\rho\left(\det\left(\mathcal{A}\right)\right)$ and $\rho\left(\det\left(\mathcal{S}\right)\right)\subseteq K\left(\mathcal{S}\right)\subseteq G\left(\mathcal{S}\right)$ by theorem~\ref{theo_main2} and lemma~\ref{lem_LXsubKXsubGX}. 
To show the second inclusions of (i) and (ii), we consider $\lambda\notin G\left(\mathcal{M}\right)\cup\rho\left(\mathrm{a}\left(f\right)\right)$ ($\lambda\notin K\left(\mathcal{M}\right)\cup\rho\left(\mathrm{a}\left(f\right)\right)$), which implies that $\mathcal{M}\left(\lambda\right)$ is sd (sdd). Since $\mathrm{a}\left(f\right)\left(\lambda\right)\neq0$, 
$\mathcal{S}\left(\lambda\right)$ is sd (sdd) by lemma~\ref{lem_sdd}.
\end{proof}
\section{Application to Ordinary Eigenvalue Problems}\label{sec_AOEP}
In this section let $R=\mathbb{C}$. We discuss the application of theorems \ref{theo_main2} and \ref{theo_conv2} to the size reduction of ordinary eigenvalue problems. We recall the definition of an eigenvalue of a complex matrix $\mathbf{X}$ as a root of $\det\left(\lambda-\mathbf{X}\right)$. The multiset of roots of a function $f\left(\lambda\right)$ is denoted $\varrho\left(f\right)$. The multiset $\varrho^s$ arises by multiplying the multipicity of each distinct element in $\varrho$ by $s$. The eigenvalue multiset of $\mathbf{X}$, its spectrum, is denoted by $\sigma\left(\mathbf{X}\right)=\varrho\left(\det\left(\lambda-\mathbf{X}\right)\right)$.

Since our reduction method gives a monic polynomial eigenvalue problem (PEP), we briefly recall how a PEP can be transformed into an ordinary eigenvalue problem.
\paragraph{Linearization of Polynomial Eigenvalue Problems}
We are interested in finding $\lambda\in\mathbb{C}$ and $\mathbf{0}\neq\mathbf{v}\in\mathbb{C}^N$ s.t.
\begin{equation}\label{eq_poly_eig}
\mathcal{L}\left(\lambda\right)\mathbf{v}=\mathbf{0}\qquad\text{with}\qquad\mathcal{L}\left(\lambda\right)=\textstyle{\text{$\sum_{k=0}^{d}$}}\mathbf{L}_k\lambda^k\text{ , }\mathbf{L}_k\in\mathbb{C}^{N\times N},
\end{equation}
in particular with the further restriction
$\mathbf{L}_d=\mathbf{I}_N$ (monoticity). The matrix $\mathcal{L}$ is also known as a lambda-matrix and the roots of $\det\left(\mathcal{L}\right)$ are called the latent roots of $\mathcal{L}$ but the term eigenvalues is rather common.

There are algorithms that work directly on this problem (cf., for instance, \cite{R73AlgNon} or \cite{V02ArnMet}). Another perhaps more widely known way to solve PEPs is via so-called \emph{linearizations}, which transform the PEP into a generalized eigenproblem $\mathbf{Z}_0\mathbf{w}=\lambda\mathbf{Z}_1\mathbf{w}$ of size $dN$ that is equivalent to it. 
A thorough treatment of linearizations can be found e.g. in \cite{AV04NewFam}, \cite{HMT06ConLin}, \cite{HMMT06SymLin}, \cite{L08LinReg} or \cite{MMMM06VecSpa}. The classical approach utilizes block companion matrices, for instance in the following form
$$\mathbf{Z}_0=\left(\begin{array}{cccc}-\mathbf{L}_{d-1}&\cdots&-\mathbf{L}_{1}&-\mathbf{L}_{0}\\\mathbf{I}&\ &\mathbf{0}&\mathbf{0}\\\ &\ddots&\ &\vdots\\\mathbf{0}&\ &\mathbf{I}&\mathbf{0}\end{array}\right),\qquad\qquad\quad$$
\begin{equation}\label{eq_comp1}
\mathbf{Z}_1=\left(\begin{array}{cccc}\mathbf{L}_{d}&\ &\ &\mathbf{0}\\\ &\mathbf{I}&\ &\ \\\ &\ &\ddots&\ \\\mathbf{0}&\ &\ &\mathbf{I}\end{array}\right)\text{ and }
\mathbf{w}=\left(\begin{array}{c}\lambda^{d-1}\mathbf{v}\\\vdots\\\lambda\mathbf{v}\\\mathbf{v}\end{array}\right).
\end{equation}
In case of a monic PEP we obtain an ordinary eigenvalue problem.

It is a disadvantage of the companion form that it in general does not reflect exploitable properties of special PEPs like those with all coefficients being hermitian.
More general linearizations allow, for instance, transformations into block symmetric generalized eigenproblems, so that in the case of hermitian lambda-matrices the linearization is hermitian, too, although moniticity is in general not preserved. Algorithms for the linear (generalized) eigenproblem may be found for example in \cite{GL96MatCom}, \cite{MS73AlgGen} or \cite{S01MatAlgII}.
\subsection{Reduction over a Schur-like complement}\label{subsec_RedSC}
We consider the eigenproblem of a complex matrix $\mathbf{M}$ with principal submatrix $\mathbf{A}$, s.t. the spectrum and an annihilating polynomial for $\mathbf{A}$ are given. The following result follows from theorem \ref{theo_main2}.
\begin{corollary}\label{theo_main}
Let $\mathbf{A}\in\mathbb{C}^{n\times n}$ be annihilated by the monic complex polynomial $$\mathrm{a}\left(x\right)=\textstyle{\text{$\sum_{k=0}^{d}$}}a_kx^k\in\mathbb{C}\left[x\right],\qquad a_d=1,\qquad\mathrm{a}\left(\mathbf{A}\right)=\mathbf{0}$$ of degree $d>0$. Let $\mathbf{B}\in\mathbb{C}^{n\times s}$, $\mathbf{C}\in\mathbb{C}^{s\times n}$, $\mathbf{D}\in\mathbb{C}^{s\times s}$,
$\mathbf{M}=\left(\begin{smallmatrix}\mathbf{A}&\mathbf{B}\\\mathbf{C}&\mathbf{D}\end{smallmatrix}\right).$
and let
\begin{equation}\label{eq_assocS}
\mathcal{S}\left(\lambda\right)=\mathrm{a}\left(\lambda\right)\left(\lambda\mathbf{I}-\mathbf{D}\right)-\mathbf{C}\mathrm{p}\left(\lambda\mathbf{I},\mathbf{A};\mathrm{a}\right)\mathbf{B}
\end{equation}
be the associated Schur-like complement.
Then, $\mathcal{S}\left(\lambda\right)$ has a linearization as a standard eigenproblem and the spectra of $\mathbf{A}$ and $\mathbf{M}$ are related through $$\sigma\left(\mathbf{M}\right)=\sigma\left(\mathbf{A}\right)+\varrho\left(\det\left(\mathcal{S}\left(\lambda\right)\right)\right)-\varrho^s\left(\mathrm{a}\left(\lambda\right)\right).$$
\end{corollary}
The matrix coefficients $\mathbf{A}_k$ of the monic lambda-matrix $\mathcal{S}\left(\lambda\right)$ are given by
\begin{align}\label{eq_coeff_S}
\mathbf{A}_k&=a_{k-1}\mathbf{I}-a_k\mathbf{D}-\textstyle{\text{$\sum_{i=0}^{d-1-k}$}}a_{k+1+i}\mathbf{C}\mathbf{A}^{i}\mathbf{B},\qquad k=1,\dots,d,\nonumber\\
\mathbf{A}_{0}&=\phantom{a_{k-1}\mathbf{I}}-a_0\mathbf{D}-\ \textstyle{\text{$\sum_{i=0}^{d-1}$}}\ a_{i+1}\mathbf{C}\mathbf{A}^{i}\mathbf{B}\quad\text{and}\quad\mathbf{A}_{d+1}=a_d\mathbf{I}.
\end{align}
The next proposition follows from corollary~\ref{theo_main} with $\lambda_0\in\sigma\left(\mathbf{A}\right)\Rightarrow\lambda_0\in\varrho\left(\mathrm{a}\right)$.
\begin{proposition}\label{pro_mult}
Let $\lambda_0$ be an eigenvalue of $\mathbf{A}$ with multiplicity $m$. If $m<s$, then $\lambda_0$ is a latent root of $\mathcal{S}\left(\lambda\right)$ with multiplicity at least $s-m$.
\end{proposition}
\paragraph{Computational costs}
Maintaining the notation of corollary~\ref{theo_main}, we define
\begin{definition}
$\kappa=\left(d+1\right)s$,
\end{definition}
which is the size of the linearized form (see above) of $\mathcal{S}\left(\lambda\right)$ . We will argue that corollary~\ref{theo_main} may offer practical numerical advantages when
\begin{equation}\label{kappalimit}
\kappa<\left(n+s\right).
\end{equation}
However, the lambda-matrix $\mathcal{S}\left(\lambda\right)$ may not inherit special properties of $\mathbf{M}$, e.g. sparseness. Therefore, \eqref{kappalimit} is more a guideline than a strict rule.
In this sense, for sake of simplicity, we assume that the practical computational cost for finding the eigenvalues of a complex matrix is roughly cubic in its size. A more rigorous analysis can be found, for instance, in~\cite{Pan99CME}.

Since, by linearization, the latent value problem of $\mathcal{S}\left(\lambda\right)$ can be transformed into an ordinary eigenvalue problem of size $\kappa$, solving it scales with $O\left(\kappa^3\right)$, which is, assuming \eqref{kappalimit}, less then $O\bigl(\left(n+s\right)^3\bigr)$, the cost for solving the eigenproblem of $\mathbf{M}$. Again assuming \eqref{kappalimit}, the construction of $\mathcal{S}\left(\lambda\right)$, i.e. the matrices $\mathbf{A}_k$, is dominated by the computation of the $\mathbf{C}\mathbf{A}^i\mathbf{B}$, $i\in\left\{0,\ldots,\left(d-1\right)\right\}$. If intermediate results are stored, the computation of those is dominated by the computation of $\mathbf{C}\mathbf{A}^{\left(d-1\right)}\mathbf{B}$. Without giving further details, we claim a simple upper bound of $O\left(\kappa n^2\right)$ for the cost of constructing $\mathcal{S}\left(\lambda\right)$, which is still less then $O\bigl(\left(n+s\right)^3\bigr)$. That cost may be further reduced if $\mathbf{A}$ is sparse or structured. Note that for small $\kappa$ the cost for constructing $\mathcal{S}\left(\lambda\right)$ is usually much larger than the cost for actually solving the corresponding latent root problem.
\paragraph{Improvements}\label{subsec_FI}
Although $\det\left(\mathcal{S}\left(\lambda\right)\right)=0$ can be transformed into an ordinary eigenvalue problem, it may be beneficial to solve the latent root problem of $\mathcal{S}\left(\lambda\right)$ without linearization using algorithms which exploit hermiticity or other structural properties (see above). In either case, it is usually an advantage to reduce $\kappa=\left(d+1\right)s$, which may be possible depending on the eigenstructure of $\mathbf{A}$. For instance, it is easy to see that $d$, the degree of the given annihilating polynomial $\mathrm{a}\left(x\right)$, can be reduced if there are roots which have higher multiplicity in the root multiset of $\mathrm{a}\left(x\right)$ than in the root multiset of the minimal polynomial, which is in a sense the optimal choice. A further example is the deflation of known eigenpairs of $\mathbf{A}$, which may allow for reducing $d$ at the expense of increasing $s$. In order to demonstrate the essential idea, we consider the case of a known eigenvector $\mathbf{v}_0$ to a simple eigenvalue $\lambda_0$.

Let $\tilde{n}=n-1$ and $\tilde{s}=s+1$. From $\mathbf{v}_0$ one can derive~\cite[pp.~11-12]{S01MatAlgII} an invertible matrix $\mathbf{Q}$ s.t.
\begin{equation}
\mathbf{Q}^{-1}\mathbf{A}\mathbf{Q}=\left(\begin{array}{cc}\mathbf{\tilde{A}}&\mathbf{y}\\\mathbf{0}&\lambda_0
\end{array}\right)\in\mathbb{C}^{\left(\tilde{n}+1\right)\times\left(\tilde{n}+1\right)}.
\end{equation}
An example for $\mathbf{Q}$ would be a Householder reflection mapping the $n$-th standard basis vector in the direction of $\mathbf{v}_0$. Since Householder reflections are rank-1-updates of the identity, they can be applied efficiently with complexity $O\left(n^2\right)$~\cite[pp.~81-82]{S01MatAlgII}.
Given $\mathrm{a}\left(x\right)$ of degree $d$ s.t. $\mathrm{a}\left(\mathbf{A}\right)=\mathbf{0}$, one derives $\tilde{\mathrm{a}}\left(x\right)\hspace{-1pt}=\hspace{-1pt}\mathrm{a}\hspace{-1pt}\left(x\right)\hspace{-1pt}/\hspace{-1pt}\left(x-\lambda_0\right)$ of degree $\tilde{d}\hspace{-1pt}=\hspace{-1pt}d-1$ s.t. $\tilde{\mathrm{a}}\bigl(\mathbf{\tilde{A}}\bigr)\hspace{-1pt}=\hspace{-1pt}\mathbf{0}$, using~\eqref{eq_(y-z)p=qy-qz}. We consider
$$
\mathbf{\tilde{M}}=\left(\begin{array}{cc}\mathbf{Q}^{-1}&\mathbf{0}\\\mathbf{0}&\mathbf{I}_s
\end{array}\right)\mathbf{M}\left(\begin{array}{cc}\mathbf{Q}&\mathbf{0}\\\mathbf{0}&\mathbf{I}_s
\end{array}\right)
=\left(\begin{array}{cc}
\mathbf{Q}^{-1}\mathbf{A}\mathbf{Q}&\mathbf{Q}^{-1}\mathbf{B}\\
\mathbf{C}\mathbf{Q}&\mathbf{D}\end{array}\right)\in\mathbb{C}^{\left(n+s\right)\times\left(n+s\right)},
$$
which may be repartitioned, without changing row and column ordering, as
\begin{equation}
\mathbf{\tilde{M}}=
\scalebox{0.7}{$
\left(\begin{array}{ccc|c|ccc}
\multicolumn{3}{c}{\multirow{3}{*}{\raisebox{0pt}{\scalebox{1.4}{$\mathbf{\tilde{A}}$}
}}}\vline&&\mathtt{x}&\cdots&\mathtt{x}\\
&&&\raisebox{0pt}{$\mathbf{y}$}&\raisebox{0pt}{$\vdots$}&\raisebox{0pt}{$\ddots$}&\raisebox{0pt}{$\vdots$}\\
&&&&\mathtt{x}&\cdots&\mathtt{x}\\\hline
0&\cdots&0&\lambda_0&\mathtt{x}&\cdots&\mathtt{x}\\\hline
\mathtt{x}&\cdots&\mathtt{x}&\mathtt{x}&\multicolumn{3}{c}{\multirow{3}{*}{\raisebox{0pt}{\scalebox{1.4}{$\mathbf{D}$}
}}}\\
\raisebox{0pt}{$\vdots$}&\raisebox{0pt}{$\ddots$}&\raisebox{0pt}{$\vdots$}&\raisebox{0pt}{$\vdots$}&&\\
\mathtt{x}&\cdots&\mathtt{x}&\mathtt{x}&&&
\end{array}\right)$}=
\left(\begin{array}{ccc}
\mathbf{\tilde{A}}&\mathbf{\tilde{B}}\\
\mathbf{\tilde{C}}&\mathbf{\tilde{D}}
\end{array}\right)
\in\mathbb{C}^{\left(\tilde{n}+\tilde{s}\right)\times\left(\tilde{n}+\tilde{s}\right)}.
\end{equation}
Summarizing, we have $\sigma\bigl(\mathbf{\tilde{M}}\bigr)=\sigma\left(\mathbf{M}\right)$, $\tilde{d}=d-1$ and $\tilde{s}=s+1$, hence 
\begin{equation}
\tilde{\kappa}=\tilde{s}\bigl(\tilde{d}+1\bigr)=\kappa+d-s.
\end{equation}
Considering $\kappa$ and $\tilde{\kappa}$ as a measure for the effort needed to solve the latent value problem of the respective associated Schur-like complements, the described transformation and repartition from $\mathbf{M}$ to $\mathbf{\tilde{M}}$ is beneficial for $d<s$.
\subsection{Reduction over an Opponent}\label{sec_RedOpp}
The next corollary  allows to exploit the spectrum and an annihilating polynomial of a complex matrix in order to determine the spectrum of a principal submatrix. It is complementary to corollary~\ref{theo_main} and follows from theorem~\ref{theo_conv2}.
\begin{corollary}\label{theo_conv}
Let $\mathbf{M}=\left(\begin{smallmatrix}\mathbf{A}&\mathbf{B}\\\mathbf{C}&\mathbf{D}\end{smallmatrix}\right)$ be a $2\times 2$ block matrix of size $\left(n+s\right)$ with $\mathbf{A}\in\mathbb{C}^{n\times n}$, $\mathbf{B}\in\mathbb{C}^{n\times s}$, $\mathbf{C}\in\mathbb{C}^{s\times n}$ and $\mathbf{D}\in\mathbb{C}^{s\times s}$. Let $\operatorname{m}\left(x;\mathbf{M}\right)=\sum_{k=0}^{d}m_kx^k$
be an annihilating polynomial for $\mathbf{M}$ of degree $d>0$, and let
\begin{equation}\label{eq_K}
\textstyle{\text{$\sum_{k=1}^{d}$}}m_k\textstyle{\text{$\sum_{i=1}^{k}$}}\lambda^{i-1}\mathbf{M}^{k-i}=\mathrm{p}\left(\lambda\mathbf{I},\mathbf{M};\mathrm{m}\right)=\left(\begin{array}{cc}\mathcal{P}_{11}\left(\lambda\right)&\mathcal{P}_{12}\left(\lambda\right)\\\mathcal{P}_{21}\left(\lambda\right)&\mathcal{K}\left(\lambda\right)\end{array}\right)
\end{equation}
be partitioned conformable to $\mathcal{M}$ s.t. $\mathcal{K}\left(\lambda\right)\in\left(\mathbb{C}\left[\lambda\right]\right)^{s\times s}$ is its lower diagonal block. Then $\sigma\left(\mathbf{A}\right)=\sigma\left(\mathbf{M}\right)
-\varrho^s\left(\operatorname{m}\left(\lambda;\mathbf{M}\right)\right)
+\varrho\left(\mathcal{K}\left(\lambda\right)\right).$
\end{corollary}
With
$\mathbf{B}_0\hspace{-2pt}=\hspace{-2pt}\left(\hspace{-3pt}\begin{array}{c}\mathbf{B}\\\mathbf{D}\end{array}\hspace{-3pt}\right)$ and $\mathbf{C}_0\hspace{-2pt}=\hspace{-2pt}\left(\hspace{-5pt}\begin{array}{cc}\mathbf{C}\hspace{-5pt}&\hspace{-5pt}\mathbf{D}\end{array}\hspace{-5pt}\right)$,
the coefficients of $\mathcal{K}\left(\lambda\right)\hspace{-2pt}=\hspace{-2pt}
\sum\limits_{k=0}^{d-1}
\mathbf{K}_k\lambda^k$ are
\begin{align}\label{eq_coeff_K}
\mathbf{K}_k\phantom{_{-1}}&=m_{k+1}\mathbf{I}+m_{k+2}\mathbf{D}+\mathbf{C}_0\left(\textstyle{\text{$\sum_{i=0}^{d-3-k}m_{k+3+i}$}}\mathbf{M}^i\right)\mathbf{B}_0,\qquad k=0,\ldots,d-2,\nonumber\\
\mathbf{K}_{d-1}&=\phantom{_{+1}}m_d\mathbf{I}.
\end{align}
\section{Vertex Perturbation of Strongly Regular Graphs}\label{sec_VPSRG}
A strongly regular graph is characterized by five parameters, $\left(n,k,\mu,\alpha,t\right)$, which determine its spectrum. $n$ is its size, the graph theoretic role of the other parameters is irrelevant, here.
For more details see~\cite{DUVAL88SRDG}. Note that, in order to avoid conflict with our notation, one parameter was relabeled.

Let $G$ be a $\left(n,k,\mu,\alpha,t\right)$-graph. Let $\mathbf{A}$ be its adjacency matrix and let $\mathbf{A}$ obey the following relations~\cite{DUVAL88SRDG}, in which $\mathbf{J}$ is the all-ones-matrix,
\begin{equation}\label{eq_SRG}
\mathbf{A}^2+\left(\mu-\alpha\right)\mathbf{A}-\left(t-\mu\right)\mathbf{I}=\mu\mathbf{J}\qquad\text{and}\qquad\mathbf{A}\mathbf{J}=\mathbf{J}\mathbf{A}=k\mathbf{J},
\end{equation}
\begin{equation}
\text{i.e.}\quad\mathbf{A}^3+\left(\mu-\alpha-k\right)\mathbf{A}^2+\left(\mu-t+k\left(\alpha-\mu\right)\right)\mathbf{A}+k\left(t-\mu\right)\mathbf{I}=\mathbf{0}.
\end{equation}
We use the annihilating polynomial
\begin{equation}
\mathrm{a}\left(x\right)=a_3x^3+a_2x^2+a_1x+a_0
\qquad\text{with}
\end{equation}
\begin{equation}
a_3=1,\quad a_2=\mu-\alpha-k,\quad a_1=\mu-t+k\left(\alpha-\mu\right),\quad a_0=k\left(t-\mu\right).
\end{equation}
\subsection{Adding Vertices to a Strongly Regular Graph}\label{subsec_AddSRG}
Let
$\mathbf{M}=\left(\begin{smallmatrix}\mathbf{A}&\mathbf{B}\\\mathbf{C}&\mathbf{D}\end{smallmatrix}\right)\in\left\{0,1\right\}^{\left(n+s\right)\times\left(n+s\right)}\subset\mathbb{C}^{\left(n+s\right)\times\left(n+s\right)}
$ be the adjacency matrix of the graph $H$, obtained by adding $s$ vertices to $G$ and connecting them arbitrarily to the initial vertices and to each other. Using \eqref{eq_assocS} and \eqref{eq_coeff_S}, we obtain the associated Schur-like complement as 
$\mathrm{S}\left(\lambda\right)=\mathbf{A}_4\lambda^4+\mathbf{A}_3\lambda^3+\mathbf{A}_2\lambda^2+\mathbf{A}_1\lambda+\mathbf{A}_0$
with
\begin{align}
\mathbf{A}_4&=a_3\mathbf{I},\nonumber\\
\mathbf{A}_3&=a_2\mathbf{I}-a_3\mathbf{D},\nonumber\\
\mathbf{A}_2&=a_1\mathbf{I}-a_2\mathbf{D}-a_3\mathbf{C}\mathbf{B},\nonumber\\
\mathbf{A}_1&=a_0\mathbf{I}-a_1\mathbf{D}-a_2\mathbf{C}\mathbf{B}-a_3\mathbf{C}\mathbf{A}\mathbf{B},\nonumber\\
\mathbf{A}_0&=\phantom{a_0\mathbf{I}}-a_0\mathbf{D}-a_1\mathbf{C}\mathbf{B}-a_2\mathbf{C}\mathbf{A}\mathbf{B}-a_3\mathbf{C}\mathbf{A}^2\mathbf{B}.
\end{align}
Assuming $s\ll n$, the computation of the coefficients $\mathbf{A}_k$ requires $s$ matrix-vector-multiplications, respectively, in order to obtain $\mathbf{A}\mathbf{B}$ and $\mathbf{A}^2\mathbf{B}=\mathbf{A}\left(\mathbf{A}\mathbf{B}\right)$, and further $3s^2$ vector-vector-multiplications for $\mathbf{C}\mathbf{B}$, $\mathbf{C}\left(\mathbf{A}\mathbf{B}\right)$, and $\mathbf{C}\left(\mathbf{A}^2\mathbf{B}\right)$. All other computations are bounded by $O\left(s^2\right)$. Therefore, the cost for constructing $\mathcal{S}\left(\lambda\right)$ is bounded by $O\left(sn^2\right)$. However, if $G$ (or its complement) is sparse, this might be further reduced significantly.
In order to find the latent roots of $\det\left(\mathcal{S}\left(\lambda\right)\right)=0$, one may exploit that it is equivalent to the ordinary (by monoticity) eigenproblem
\begin{equation}
\det\left(\left(\begin{array}{cccc}\mathbf{A}_{3}&\mathbf{A}_{2}&\mathbf{A}_1&\mathbf{A}_0\\\mathbf{I}&\mathbf{0}&\mathbf{0}&\mathbf{0}\\\mathbf{0}&\mathbf{I}&\mathbf{0}&\mathbf{0}\\\mathbf{0}&\mathbf{0}&\mathbf{I}&\mathbf{0}\end{array}\right)-\lambda\left(\begin{array}{cccc}\mathbf{I}&\mathbf{0}&\mathbf{0}&\mathbf{0}\\\mathbf{0}&\mathbf{I}&\mathbf{0}&\mathbf{0}\\\mathbf{0}&\mathbf{0}&\mathbf{I}&\mathbf{0}\\\mathbf{0}&\mathbf{0}&\mathbf{0}&\mathbf{I}\end{array}\right)\right)=0
\end{equation}
of size $\kappa=4s$ (independent of $n$). Since $k$ is always a simple eigenvalue of $\mathbf{A}$ with constant eigenvector,~\cite{DUVAL88SRDG} or \eqref{eq_SRG}, we already know, by proposition~\ref{pro_mult}, that, for s>1, $k$ must be a latent root of $S\left(\lambda\right)$ with multiplicity at least $s-1$. Furthermore, according to subsection~\ref{subsec_FI}, by deflation of the eigenvalue $k$ in $\mathbf{A}$, we may efficiently transform and repartition $\mathbf{M}$ into 
\begin{equation}
\mathbf{\tilde{M}}=\mathbf{\tilde{Q}}^{-1}\mathbf{M}\mathbf{\tilde{Q}}=\left(\begin{array}{cc}\mathbf{\tilde{A}}&\mathbf{\tilde{B}}\\\mathbf{\tilde{C}}&\mathbf{\tilde{D}}\end{array}\right)\in\mathbb{C}^{\left(\left(n-1\right)+\left(s+1\right)\right)\times\left(\left(n-1\right)+\left(s+1\right)\right)}
\end{equation}
s.t. $\mathbf{\tilde{A}}\hspace{-1pt}\in\hspace{-1pt}\mathbb{C}^{\left(n-1\right)\times\left(n-1\right)}$ is annihilated by $\tilde{\mathrm{a}}\left(x\right)\hspace{-2pt}=\hspace{-2pt}x^2+\left(\mu-\alpha\right)x+\left(\mu-t\right)$. The associated Schur-like complement of $\mathbf{\tilde{A}}$ admits a linearization of size $\tilde{\kappa}\hspace{-1pt}=\hspace{-1pt}3\hspace{-1pt}\left(s+1\right)$.
\subsection{Removing Vertices from a Strongly Regular Graph}\label{subsec_RemSRG}
Employing the block partition $\mathbf{A}=\left(\begin{smallmatrix}\mathbf{R}&\mathbf{S}\\\mathbf{T}&\mathbf{U}\end{smallmatrix}\right)\in\left\{0,1\right\}^{\left(\tilde{n}+s\right)\times\left(\tilde{n}+s\right)}\subset\mathbb{C}^{\left(\tilde{n}+s\right)\times\left(\tilde{n}+s\right)}$, the principal submatrix $\mathbf{R}$ is the adjacency matrix of the graph $H'$ of size $\tilde{n}=n-s$, obtained by removing the last $s$ vertices of $G$.
According to subsection~\ref{sec_RedOpp}, the associated opponent $\mathcal{K}\hspace{-1pt}\left(\lambda\right)\hspace{-2pt}=\hspace{-2pt}\mathbf{K}_2\lambda^2+\mathbf{K}_1\lambda+\mathbf{K}_0$, which is the lower right block of
\begin{equation}\label{eq_Ksrg}
\textstyle{\text{$\sum_{k=1}^{3}$}}a_k\textstyle{\text{$\sum_{i=1}^{k}$}}\lambda^{i-1}\mathbf{A}^{k-i}=\mathrm{p}\left(\lambda\mathbf{I},\mathbf{A};\mathrm{a}\right)=\left(\begin{array}{cc}\mathcal{P}_{R}\left(\lambda\right)&\mathcal{P}_{S}\left(\lambda\right)\\\mathcal{P}_{T}\left(\lambda\right)&\mathcal{K}\left(\lambda\right)\end{array}\right),
\end{equation}
is given, using~\eqref{eq_coeff_K} and abbreviating $\mathbf{S}_0=\left(\begin{array}{c}\mathbf{S}\\\mathbf{U}\end{array}\right)$ and $\mathbf{T}_0=\left(\begin{array}{cc}\mathbf{T}&\mathbf{U}\end{array}\right)$, by
\vspace{-10pt}
\begin{align}
\mathbf{K}_2&=a_3\mathbf{I},\nonumber\\
\mathbf{K}_1&=a_2\mathbf{I}+a_3\mathbf{U},\nonumber\\
\mathbf{K}_0&=a_1\mathbf{I}+a_2\mathbf{U}+a_3\mathbf{T}_0\mathbf{S}_0.
\end{align}
In case of $s\ll n$, the cost for constructing $\mathcal{K}\left(\lambda\right)$ is dominated by the $O\left(s^2n\right)$ cost for the product $\mathbf{S}_0\mathbf{T}_0$.
The latent values of $\mathcal{K}\left(\lambda\right)$ are the eigenvalues of
\begin{equation}
\det\left(\left(\begin{array}{ccc}\mathbf{K}_{1}&\mathbf{K}_{0}\\\mathbf{I}&\mathbf{0}\end{array}\right)-\lambda\left(\begin{array}{ccc}\mathbf{I}&\mathbf{0}\\\mathbf{0}&\mathbf{I}\end{array}\right)\right)=0,
\end{equation}
an ordinary eigenproblem of size $2s$. For undirected strongly regular graphs $G$ and $s=1$, we have $\mathbf{U}=0$ and $\mathbf{T}_0\mathbf{S}_0=k$ independent of the vertex ordering, which implies the well known fact that all vertex deleted subgraphs of an undirected strongly regular graph have the same spectrum.
\bibliography{RelEigAdjSch}
\end{document}